\documentclass[11pt, a4paper, twoside]{amsart}

\pdfoutput=1

\usepackage{tikz}
\usepackage{amsmath}
\usepackage{amsfonts}
\usepackage{amssymb}
\usepackage{latexsym}
\usepackage{mathtools}
\usepackage{subcaption}
\usepackage{epstopdf}
\mathtoolsset{showonlyrefs}

\newcommand{\GL}{\mathrm{GL}} 

\newcommand{\conv}{\mathrm{conv}\,} 

\newcommand{\vol}{\mathrm{vol}\,} 
\newcommand{\mvol}{\mathrm{vol}\,} 

\newcommand{\Lat}{\mathcal{L}} 

\newcommand{\Ksn}{{\mathcal K}_{(s)}^n} 

\newcommand{\R}{\mathbb{R}} 

\newcommand{\N}{\mathbb{N}} 
\newcommand{\Z}{\mathbb{Z}} 

\newcommand{\disuni}{\mathbin{\setbox0\hbox{$\bigcup$}\rlap{\copy0}\raise.3%
  \ht0\hbox to \wd0{\hfil$\cdot$\hfil}}}
  
\newcommand{\ip}[2]{\left\langle #1,#2\right\rangle}

\newcommand{\dual}[1]{{#1}^\star}
\newcommand{\ov}{\overline}

\newcommand{\ban}{B_n}

\newcommand{\trans}{\intercal}

\newcommand{\va}{{\boldsymbol a}}
\newcommand{\vb}{{\boldsymbol b}}
\newcommand{\vc}{{\boldsymbol c}}
\newcommand{\ve}{{\boldsymbol e}}

\newcommand{\vu}{{\boldsymbol u}}

\newcommand{\vx}{{\boldsymbol x}}
\newcommand{\vy}{{\boldsymbol y}}
\newcommand{\vz}{{\boldsymbol z}}

\newcommand{\vone}{{\boldsymbol 1}}

\newcommand{\vnull}{{\boldsymbol 0}}

\newcommand{\gap}{\rm{GAP}}

\newcommand{\gapP}{{\mathrm{P}}}

\newcommand{\symgap}[2]{{\gapP\left(#1, #2\right)}}

\newcommand{\landauO}[1]{{\operatorname{O}\left(#1\right)}}

\definecolor{196C}{HTML}{ECC7CD}	
\definecolor{197C}{HTML}{E89CAE}	
\definecolor{198C}{HTML}{DF4661}	
\definecolor{199C}{HTML}{D50032}	
\definecolor{200C}{HTML}{BA0C2F}	
\definecolor{201C}{HTML}{9D2235}	
\definecolor{202C}{HTML}{862633}	

\definecolor{427C}{RGB}{208,211,212}
\definecolor{428C}{RGB}{193,198,200}
\definecolor{429C}{RGB}{162,170,173}
\definecolor{430C}{RGB}{124,135,142}
\definecolor{431C}{RGB}{91,103,112}
\definecolor{432C}{RGB}{51,63,72}
\definecolor{433C}{RGB}{29,37,45}

\definecolor{7681C}{HTML}{94A9CB}	
\definecolor{7682C}{HTML}{6787B7}	
\definecolor{7683C}{HTML}{426DA9}	
\definecolor{7684C}{HTML}{385E9D}	
\definecolor{7685C}{HTML}{2C5697}	
\definecolor{7686C}{HTML}{1D4F91}	
\definecolor{7687C}{HTML}{1D428A}	

\definecolor{7695C}{HTML}{7BA7BC}	
\definecolor{7696C}{HTML}{6399AE}	
\definecolor{7697C}{HTML}{4E87A0}	
\definecolor{7698C}{HTML}{41748D}	
\definecolor{7699C}{HTML}{34657F}	
\definecolor{7700C}{HTML}{165C7D}	
\definecolor{7701C}{HTML}{005776}	

\definecolor{3248C}{HTML}{6DCDB8} 
\definecolor{3258C}{HTML}{49C5B1} 
\definecolor{3268C}{HTML}{00AB8E} 
\definecolor{3278C}{HTML}{009B77} 
\definecolor{3288C}{HTML}{008264} 
\definecolor{3298C}{HTML}{006A52} 
\definecolor{3308C}{HTML}{034638} 

\definecolor{310C}{HTML}{6AD1E3}	
\definecolor{311C}{HTML}{05C3DE}	
\definecolor{312C}{HTML}{00A9CE}	
\definecolor{313C}{HTML}{0092BC}	
\definecolor{314C}{HTML}{007FA3}	
\definecolor{315C}{HTML}{00677F}	
\definecolor{316C}{HTML}{004851}	

\definecolor{705C}{HTML}{F5DADF}	
\definecolor{706C}{HTML}{F7CED7}	
\definecolor{707C}{HTML}{F9B5C4}	
\definecolor{708C}{HTML}{F890A5}	
\definecolor{709C}{HTML}{EF6079}	
\definecolor{710C}{HTML}{E03E52}	
\definecolor{711C}{HTML}{CB2C30}	

\definecolor{162C}{HTML}{FFBE9F}	
\definecolor{163C}{HTML}{FF9D6E}	
\definecolor{164C}{HTML}{FF7F41}	
\definecolor{165C}{HTML}{FF671F}	
\definecolor{166C}{HTML}{E35205}    
\definecolor{167C}{HTML}{BE531C}	
\definecolor{168C}{HTML}{73381D}	

\colorlet{mygray}{430C}

\colorlet{thesiscolor1}{427C}
\colorlet{thesiscolor2}{428C}
\colorlet{thesiscolor3}{429C}
\colorlet{thesiscolor4}{430C}
\colorlet{thesiscolor5}{431C}
\colorlet{thesiscolor6}{432C}
\colorlet{thesiscolor7}{433C}

\colorlet{thesis_2dshade}{429C}
\colorlet{thesis_2dcontrast}{199C}
\colorlet{thesis_2doutline}{432C}
\colorlet{thesis_3dshade1}{thesiscolor1}
\colorlet{thesis_3dshade2}{thesiscolor4}
\colorlet{thesis_3dshade3}{thesiscolor6}

\usepackage{hyperref}
\hypersetup{backref, pdfpagemode=FullScreen, colorlinks=true,
  citecolor=200C, linkcolor=cyan, urlcolor=7685C}

\usepackage{version} 

\theoremstyle{change}
\newtheorem{theorem}{Theorem}[section]
\newtheorem*{theorem*}{Theorem}
\newtheorem{lemma}[theorem]{Lemma}

\newtheorem{proposition}[theorem]{Proposition}

\newtheorem{remark}[theorem]{Remark}

\numberwithin{equation}{section}
\numberwithin{figure}{section}

\begin{document}

\title{On a discrete John-type theorem} 
\author{S\"oren Lennart Berg}
\author{Martin Henk}
\address{Institut f\"ur Mathematik, Sekr. MA 4-1, Technische
  Universit\"at Berlin, Stra{\ss}e des 17. Juni 136, D-10623 Berlin,
  Germany}
\email{berg@math.tu-berlin.de, henk@math.tu-berlin.de}
\thanks{This paper contains some material of the PhD thesis of the
  first named author. }

\begin{abstract} As a discrete counterpart to the classical John theorem on the
  approximation of (symmetric) $n$-dimensional convex bodies $K$ by ellipsoids, Tao and Vu
  introduced so called generalized arithmetic progressions 
  $\gapP(A,\vb)\subset\Z^n$ 
   in order to
  cover (many of) the  lattice points inside a convex body by a simple
  geometric structure. Among others, they proved that there exists a  generalized arithmetic progressions 
  $\gapP(A,\vb)$ such that $\gapP(A,\vb)\subset K\cap\Z^n\subset
  \gapP(A,O(n)^{3n/2}\vb)$. Here we show that this bound can be
  lowered to $n^{O(\ln n)}$
  and study some genereal properties of so called unimodular  generalized arithmetic progressions.
\end{abstract} 

\maketitle

\section{Introduction} 
Let  $\Ksn$ be the set of all $o$-symmetric convex bodies in $\R^n$,
i.e., $K\in\Ksn$ is a compact convex set in $\R^n$ with non-empty
interior and $K=-K$. By $\ban\in\Ksn$ we denote the $n$-dimensional unit
 ball, i.e., $\ban=\{\vx\in\R^n : \ip{\vx}{\vx}\leq 1\}$, where
 $\ip{\cdot}{\cdot}$ is the standard inner product.

For the family $\Ksn$ of $o$-symmetric convex bodies in $\R^n$, 
John's (ellipsoid) theorem states that there exists an ellipsoid
$\mathcal{E}\in\Ksn$ such that (see, e.g., \cite[Theorem
2.1.3]{ArtsteinAvidan:2015hq}, \cite[Theorem 10.12.2]{Schneider:2014td})
\begin{equation}
    \mathcal{E} \subseteq K\subseteq \sqrt{n}\, \mathcal{E}.  
\label{eq:inclusion_john}
\end{equation}
It turns out that the volume maximal ellipsoid contained in $K$ gives
the desired approximation, and in the non-symmetric (or general) case
the factor $\sqrt{n}$  has to be replaced by $n$ (after a suitable
translation of $K$).    

This theorem has numerous applications in Convex Geometry or in the
local theory of Banach spaces (see, e.g.,
\cite{ArtsteinAvidan:2015hq}, \cite{Schneider:2014td}), as it allows
to get a first quick 
estimate on the value $f(K)$ of any homogenous and monotone
functional $f$  on $\Ksn$ by the value of the functional at ellipsoids. 
For instance, if  $\mvol()$ denotes the $n$-dimensional volume, i.e.,
$n$-dimensional Lebesgue measure, than \eqref{eq:inclusion_john}
implies that for $K\in\Ksn$ there exists an ellipsoid $\mathcal{E}$
such that 
\begin{equation}
    \mvol(\mathcal{E}) \leq\mvol(K)\leq  n^{\frac{n}{2}}\mvol(\mathcal{E}).  
\label{eq:volume_john}
\end{equation}
In particular, the volume of an ellipsoid can easily be evaluated as
  $\mathcal{E} = A\,\ban$ for some $A\in\GL(n,\R)$, and thus
  $\vol(\mathcal{E})=|\det A|\,\vol(B_n)$.  

In \cite{Tao:2006vs}, Tao and Vu started to study a discrete version of John's
theorem where the aim of the approximation is the set of lattice
points in $K$, i.e., the set $K\cap\Z^n$.  The approximation itself is
carried out not by lattice points in ellipsoids, which are hard to
control or to compute, but by a so called symmetric generalized
arithmetic progression 
({\gap} for short)
\begin{equation*} 
 \symgap{A}{\vb}=\{A\, \vz : \vz\in\Z^n,\, |z_i|\leq b_i, 1\leq i\leq n\},  
\end{equation*} 
where $A\in \Z^{n\times n}$, $\det A\ne 0$,  and $\vb\in \R^n$. Hence,
$\symgap{A}{\vb}$ are  the lattice points of the lattice $A\Z^n$ in the 
parallelepiped $\sum_{i=1}^n\conv\{-b_i\va_i,b_i\va_i\}$, where  $\va_i$ is
the $i$th column  of $A$ and $\conv$ denotes the convex hull.  

By improving on an earlier  result from \cite[Lemma 3.36]{Tao:2006vs}, 
 Tao and Vu proved in \cite{Tao:2008ul}

\begin{theorem*}[\protect{\cite[Theorem 1.6]{Tao:2008ul}}] Let
  $K\in\Ksn$. Then  there  exists a {\gap} $\symgap{A}{\vb}\subset K$  
    such that  
    \begin{equation}\label{eq:taovueq1}
      \begin{split} 
       {\rm i)}&\,\, 
        K\cap\Z^n 
        \subset \symgap{A}{\landauO{n}^{3n/2}\vb}, \\
        {\rm ii)}&\,\,
        |K\cap\Z^n|
       < \landauO{n}^{7n/2}|\symgap{A}{\vb}|.
        \end{split} 
    \end{equation}
\end{theorem*} 
Here, for a finite set $C$  we denote by $|C|$ its
cardinality, and  observe that $|\gapP(a,\vb)|=\prod_{i=1}^n(2 \lfloor
b_i\rfloor+1)$ can be easily computed.  Obviously, i) and ii) of the
theorem above may be regarded as discrete
counterparts to \eqref{eq:inclusion_john} and \eqref{eq:volume_john}.
 
A first qualitative version of such a theorem, i.e., without mentioning
 explicit constants, is contained in the paper \cite[Theorem
3]{Barany:1992p12355}. 
Here we show
\begin{theorem}
\label{thm:discretejohnimproved}    Let $K\in\Ksn$. 
\begin{enumerate} 
\item
There
    exists a {\gap} $\symgap{A}{\vb}\subset K$ such that
\begin{equation}
     K\cap\Z^n \subset \symgap{A}{n^{\landauO{\ln n}} \
    \vb}. 
\label{eq:inclusion}
\end{equation}
\item There exists a {\gap} $\symgap{A}{\vb}\subset K$ such that
\begin{equation}\label{eq:cardinality}
         |K\cap\Z^n|
       < \landauO{n}^{n}| \symgap{A}{\vb}|.
    \end{equation}
\end{enumerate}
\end{theorem}
In comparison to the volume case (John's ellipsoid) 
 a  {\gap}
contained in $K\in\Ksn$  which is optimal for the cardinality bound
\eqref{eq:cardinality}, i.e., covering most of the lattice point in
$K$,  
does not need to be optimal for the inclusion bound \eqref{eq:inclusion}  as well.
We will give an example of this occurence in Proposition \ref{prop:cardinlu}. In
fact, also the two {\gap}s leading to the bounds in \eqref{eq:inclusion} and
\eqref{eq:cardinality} are different (in general). 

Regarding a {\gap} $\gapP(A,\vb)$ which is simultaneously good with respect to
inclusion and cardinality we have the following slight improvement on
the above theorem of Tao and Vu. 

\begin{theorem} Let $K\in\Ksn$. Then  there  exists a
    {\gap} $\gapP(A,\vb)\subset K$  
    such that  
    \begin{equation}
      \begin{split} 
       {\rm i)}&\,\, 
        K\cap\Z^n 
        \subset \symgap{A}{\landauO{n}^{2 n/\ln n}\vb}, \\
        {\rm ii)}&\,\,
        |K\cap\Z^n|
       < \landauO{n}^{2n}|\symgap{A}{\vb}|.
        \end{split} 
      \end{equation}
\label{thm:onegap}      
\end{theorem}

For unconditional convex bodies $K\in\Ksn$, i.e., $K$ is symmetric to
all coordinate hyperlanes, the inclusion bound can be made linear. 

\begin{proposition}
    Let $K\in\Ksn$ be an unconditional convex body.  Then there
    exists a {\gap} $\gapP(A,\vb)\subset K$ such that
    \begin{equation}
      \begin{split} 
       {\rm i)}&\,\, 
        K\cap\Z^n 
        \subseteq \symgap{A}{n\,\vb}, \\
        {\rm ii)}&\,\,
        |K\cap\Z^n|
       < O(n)^n|\symgap{A}{\vb}|.
        \end{split} 
    \end{equation}
\label{prop:unconditional}
\end{proposition}

As we will point out in Proposition \ref{prop:lowerbounds}, the linear inclusion
bound  in Proposition \ref{prop:unconditional} is essentially best
possible, and it might be even true that the general bound of order $n^{\landauO{\ln n}}$
in
\eqref{eq:inclusion} can be replaced by a linear or polynomial bound in
$n$.

The paper is organized as follows. In the second section we introduce
and collect some basic properties of {\gap}s approximating the lattice
points in symmetric convex bodies. In turns out that {\gap}s where the
columns of $A$ form a lattice basis of $\Z^n$ are of particular interest
and we study them in Section 3.  Finally,  Section 4 contains the proof of the
theorems and of the proposition above.

\section{Preliminaries and {\gap}s}\label{sec:preliminaries} 
For the proof of Theorem  \ref{thm:discretejohnimproved} 
it is more
convenient to introduce {\gap}s for general lattices
$\Lambda\subset\R^n$, i.e., $\Lambda=B\,Z^n$, $B\in\R^{n\times n}$  with $\det B\ne 0$. Let
$\Lat^n$ be the set of all these lattices. Following \cite{Tao:2008ul}, and
adapting their definition to our special geometric situation, we call
for a matrix $A\in\R^{n\times n}$ with columns $\va_i\in\Lambda$,
$1\leq i\leq n$, and for $\vb\in\R^n_{>0}$ the set of lattice points in
$\Lambda$ given by 
\begin{equation*}
   \gapP(A,\vb)=\{A\vz : -\vb\leq \vz\leq \vb, \vz\in\Z^n \}
\end{equation*}
a generalized symmetric arithmetic progression with respect to
$\Lambda$, for short, just {\gap}.

Actually, Tao and Vu defined {\gap}s more generally, namely, for general
$n\times m$ matrices $A$. 
 In our
geometric setting, however, this  would make the inclusion
bound needless as $A$ may consist of all (up to $\pm$) lattice
points in $K\in \Ksn$ and by letting $\vb=(1-\epsilon)\vone$, where $\vone$ is
the appropriate  all $1$-vector and $\epsilon$ an arbitrary positive
number less than 1, we get 
\begin{equation*}
  \{\vnull\} = P(A,\vb) \subset K\cap\Z^n \subset
  P(A,(1-\epsilon)^{-1}\vb) 
\end{equation*}
Moreover, Tao and Vu were mainly interested in so called infinitely
proper {\gap}s which here means $m=\mathrm{rank}(A)$, and so we restrict the
definition to the case $A\in\R^{n\times n}$, $\det A\ne 0$.

The size or cardinality of a {\gap} $\gapP(A,\vb)$  is given
\begin{equation*}
    |\gapP(A,\vb)|= \prod_{i=1}^n (2\lfloor b_i\rfloor+1),
\end{equation*}
where $\lfloor\cdot\rfloor$ denotes the floor function. In general, for
a vector $\vb\in\R^n$ we denote by $\lfloor \vb\rfloor=(\lfloor
b_1\rfloor,\dots,\lfloor b_n\rfloor)^\trans$ its integral part. 
With $\gapP_\R(A,\vb)$ we denote the parallelepiped 
\begin{equation*}
  \gapP_\R(A,\vb) = \{A\vx : -\vb\leq \vx\leq \vb,\,\vx\in\R^n \} = \sum_{i=1}^n\conv\{-b_i\va_i,b_i\va_i\}
\end{equation*}
associated to the {\gap} $\gapP(A,\vb)$. Observe that
\begin{equation}
  \label{eq:5}
    \gapP_\R(A,\lfloor \vb\rfloor)=\conv \gapP(A,\vb).
  \end{equation}
  
Whenever we are in interested in a {\gap} $\gapP(A,\vb)$ covering most of
the lattice points in a convex body, i.e., a {\gap} which is optimal with
respect to the cardinality bound, then it suffices to assume
$\vb\in\N^n$. However, for an optimal {\gap} with respect to the
inclusion bound it might be essential to consider non-integral vectors
$\vb\in\R^n_{>0}$. 
This is also reflected by the next example showing that those {\gap}s yielding an optimal
cardinality bound can be different from those leading to an optimal
inclusion bound.

\begin{proposition} Let $n\geq 2$.  There exists a $K\in\Ksn$ such
  that   any {\gap} $\gapP(A,\vb)\subset K$ covering most of the lattice
  points of $K$ is not an optimal {\gap} with respect to inclusions,
  i.e., there exists another {\gap} $\gapP(\ov A,\ov \vb)\subset K$ such that 
for any $t> 1$ with $K\cap\Z^n\subseteq \gapP(A,t\,\vb)$ there exits a
$\ov t<t$ with $K\cap\Z^n\subseteq \gapP(\ov A,\ov t\,\ov\vb)$.
\label{prop:cardinlu}
\end{proposition} 

\begin{proof} We start with dimension $2$, and let $K=\conv\{\pm (3,0)^\trans,  \pm (-3,1)^\trans, \allowbreak\pm
(-1,1)^\trans\}$ (see Figure \ref{fig:planemaxfactor}).

 \begin{figure}[hbt]
        \centering
        \begin{subfigure}{.58\textwidth}
              \centering
              \begin{tikzpicture}[scale=0.6]
                \clip (-5.2,-2.2) rectangle (5.2,2.2);

                \coordinate (v1) at (-3,1);
                \coordinate (v2) at (-1,1);
                \coordinate (v3) at (3,0);
                \coordinate (v4) at (3,-1);
                \coordinate (v5) at (1,-1);
                \coordinate (v6) at (-3,0);
                \draw[thick] (v1)--(v2)--(v3)--(v4)--(v5)--(v6)--(v1);

                \draw[dashed,thin,fill=thesis_2dshade, fill opacity=0.5]
                    (-3,1)--(-1,1)--(3,-1)--(1,-1)--cycle;
                \draw[dashed,thin]
(-5.4,1.8)--(-1.8,1.8)--(5.4,-1.8)--(1.8,-1.8)--cycle;

                \draw [thick, thesis_2dcontrast,-latex,opacity=0.8] (0,0) -- (1,0);
                \draw [thick, thesis_2dcontrast,-latex,opacity=0.8] (0,0) -- (-2,1);

                \foreach \x in {-7,-6,...,7}{
                  \foreach \y in {-7,-6,...,7}{
                    \node[draw,circle,inner sep=0.75pt,fill] at (\x,\y) {};
                  }
                }
            \end{tikzpicture}
        \end{subfigure}
        \begin{subfigure}{.38\textwidth}
              \centering
              \begin{tikzpicture}[scale=0.6]
                \clip (-3.2,-2.2) rectangle (3.2,2.2);

                \coordinate (v1) at (-3,1);
                \coordinate (v2) at (-1,1);
                \coordinate (v3) at (3,0);
                \coordinate (v4) at (3,-1);
                \coordinate (v5) at (1,-1);
                \coordinate (v6) at (-3,0);
                \draw[thick] (v1)--(v2)--(v3)--(v4)--(v5)--(v6)--(v1);

                \draw[dashed,thin,fill=thesis_2dshade, fill opacity=0.5]
                    (-3,0.8)--(3,0.8)--(3,-0.8)--(-3,-0.8)--cycle;

                \draw [thick, thesis_2dcontrast,-latex,opacity=0.8] (0,0) -- (1,0);
                \draw [thick, thesis_2dcontrast,-latex,opacity=0.8] (0,0) -- (0,1);

                \foreach \x in {-7,-6,...,7}{
                  \foreach \y in {-7,-6,...,7}{
                    \node[draw,circle,inner sep=0.75pt,fill] at (\x,\y) {};
                  }
                }
            \end{tikzpicture}
        \end{subfigure}
        \caption{Different optimal {\gap}s}
         \label{fig:planemaxfactor}
    \end{figure}

We will argue that  an optimal cardinality {\gap} $P(A,\vb)\subseteq
K$ contains 9 out the 13 lattice points in $K$. To this end we may
assume  that the columns $\va_i$ of $A$ belong to $K$, i.e.,
$\va_i\in K$ and $\vb\geq\vone$. Otherwise, we could only cover
lattice points on a line which would be at most 7. Since for all
$\vx\in K$ we have $|x_2|\leq 1$, and since also the sum $\va_1+\va_2$
has to belong to $K$,  
there is at most one column  $\va_i$ of $A$ having a non-zero last coordinate. 

If there would be none, then again only the $7$ points with last coordinate $0$ could
be covered. 

Next assume that $\va_2$ is the vector having last coordinate
non-zero and let $\va_1$ be the vector with last coordinate $0$. The only possibility so
that $\va_1\pm\va_2$ belong to $K$ are (up to sign) the one 
depicted in the left figure, i.e.,  $\va_1=(1,0)^\trans$ and
$\va_2=(-2,1)^\trans$, and for any $\vb$ with  $1\leq b_i<2$, $i=1,2$,
the {\gap} $\gapP(A,\vb)$ covers 9 out of the 13 lattice points of $K$.     
Hence, the {\gap}s covering the maximal amount of lattice points of $K$
are -- up to $\pm$ and permutation of the columns of $A$ -- given by  $\gapP(A,\vb)$ for any $\vb$ with  $1\leq b_i<2$, $i=1,2$.   
Sine $(3,0)^\trans\in K$, we observe that 
in order to cover all the points of 
$K\cap\Z^2$ by $\gapP(A,t\,\vb)$ we must have $t>3/2$.

On the other hand, if we take for the columns of $\ov A$ the vectors
$(1,0)^\trans$ and $(0,1)^\trans$ and setting $\ov\vb=(3,1-\epsilon)$ we
get $|\gapP(\ov A,\ov \vb)|=7$,  but  $K\cap\Z^2 \subset
  P(\ov A,(1-\epsilon)^{-1}\ov \vb)$ for any $\epsilon\in (0,1)$ (cf.~right
  hand side picture in Figure \ref{fig:planemaxfactor}).

This verifies the assertion in the plane. 
By building successively prisms over $Q$ the example can be extended to
all dimensions.

\end{proof}

\section{Unimodular {\gap}s}
Without loss of generality we consider here only the case
$\Lambda=\Z^n$. The group of all unimodular matrices, i.e., integral $n\times n$-matrices of
determinant $\pm 1$,  is denoted by
$\GL(n,\Z)$; it consists of all lattices bases of $\Z^n$. 
Apparently, if $K\cap \Z^n$ contains a lattice basis of $\Z^n$ and
$K\cap\Z^n\subseteq \gapP(A,\vb)$ then $A\in \GL(n,\Z)$. This basically
shows that for the inclusion bound it suffices to consider {\gap}s $\gapP(U,\vb)$ with $U\in
\GL(n,\Z)$. We will call such a {\gap}, an unimodular {\gap}. 

\begin{proposition} Let $c=c(n)\in\R_{>0}$ be a constant depending on
  $n$. The following statements are equivalent.
  \begin{enumerate}
    \item For every $K\in\Ksn$ there exists a {\gap} $\gapP(A,\vb)\subset
      K$ such that $K\cap\Z^n\subset \gapP(A,c\,\vb)$.
    \item   For every $K\in\Ksn$ there exists an unimodular  {\gap}  $\gapP(U,\vb)\subset
      K$ such that $K\cap\Z^n\subset \gapP(U,c\,\vb)$.
 \end{enumerate} 
\label{prop:1} 
\end{proposition} 
\begin{proof} Obviously, we only have to show that i) implies ii). To
  this end let
  $l\in\N$  such that $l\,K$ contains a basis of $\Z^n$. By assumption
  there exists a {\gap} $P(U,\vb)\subset lK$ such that $lK\cap\Z^n\subseteq
  P(U,c\,\vb)$ and since $lK$ contains a basis of $\Z^n$ we have
  $U\in\GL(n,\Z)$. Next we claim that
  \begin{equation}
    \label{eq:1}
     P(U,l^{-1}\vb)\subseteq K\cap\Z^n\subseteq P(U,c\,l^{-1}\vb).
   \end{equation}
   Let $\vu\in P(U,l^{-1}\vb)$. Then there exists a $\vz\in\Z^n$ with
   $\vu=U\vz$ and $-l^{-1}\vb\leq \vz\leq l^{-1}\vb$. Thus
   $l\vu=U\,l\vz$ and since $l\,\vz \in\Z^n$ we get $l\vu\in
   \gapP(U,\vb)\subset l\,K$. Hence $\vu\in K\cap\Z^n$ which shows the
   first inclusion in \eqref{eq:1}. For the second let $\va\in
   K\cap\Z^n$. Then $l\,\va\in l\,K\cap\Z^n\subseteq P(U,c\,\vb)$ and so
   there exists a $\vz\in\Z^n$ with $-c\,\vb\leq\vz\leq c\,\vb$ with
   $l\,\va=U\vz$. Hence, $\va = U\, l^{-1}\vz$ and since $U\in
   \GL(n,\Z)$ we conclude $l^{-1}\vz\in\Z^n$ which shows $\va\in  P(U,c\,l^{-1}\vb)$.
\end{proof} 

Next we want to point out a relation between {\gap}s and approximations
of a convex body by an ``unimodular'' parallelepiped
$\gapP_\R(U,\vu)$, $U\in\GL(n,\Z)$. To this we first
note that 
\begin{lemma} Let $K\in \Ksn$ containing $n$ linearly independent
  points $\beta\va_i$ with $\beta\in\R_{>0}$ and
  $\va_i\in\Z^n$, $1\leq i\leq n$.  Then for any unimodular {\gap}
  $\gapP(U,\vu)$ with $K\subseteq \gapP_\R(U,\vu)$ we have $u_i\geq \beta$,
  $1\leq i\leq n$.
\label{lem:1} 
\end{lemma} 
\begin{proof} Let $\beta\va_i= U\vx_i$ with
  $-\vu\leq\vx_i\leq\vu$, $\vx_i\in\R^n$. Since $U\in\GL(n,\Z)$ we get
  $\vx_i\in\beta\Z^n$, which shows that for each non-zero coordinate
  $j$, say, of $\vx_i$  we have $u_j\geq \beta$. Since
  $\vx_1,\dots,\vx_n$ are linearly independent for each coordinate $k$
  we can find a vector $\vx_l$ whose $k$th coordinate is non-zero.      
\end{proof} 

Observe, for an unimodular {\gap}
$\gapP(U,\vu)$ we
have $\gapP(U,\vu)= \gapP_\R(U,\vu)\cap\Z^n$. 

\begin{proposition} 
Let $c=c(n)\in\R_{>0}$ be a constant depending on
  $n$. The following statements are equivalent.
\begin{enumerate} 
  \item For every $K\in\Ksn$ there exists a {\gap} $\gapP(A,\vb)\subset
      K$ such that $K\cap\Z^n\subseteq \gapP(A,c\,\vb)$.
    \item   For every $K\in\Ksn$ there exists an unimodular  {\gap}  $\gapP(U,\vu)\subset
      K$ such that 
\begin{equation*}
    \gapP_\R(U,\vu)\subseteq K \subset \gapP_\R(U,c\,\vu). 
\end{equation*} 
\end{enumerate} 
\end{proposition} 
\begin{proof}  We start with the implication i) implies ii). Let
  $\epsilon>0$, and let $Q\subseteq K$ be a
  $o$-symmetric rational
  polytope with $K\subset (1+\epsilon)Q$ (see, e.g., \cite[Theorem 1.8.19]{Schneider:2014td}). Moreover, let $m\in\N$ such
  i) $mQ$ is an integral polytope, i.e., all vertices are in $\Z^n$, and
  ii) $mQ$ contains the scaled unit vectors $c(1+c/\epsilon)\ve_i$,
  $1\leq i\leq n$. In view of Proposition \ref{prop:1}  there exists
  an unimodular {\gap} $\gapP(U,\vu)$ such that 
  \begin{equation*}
      \gapP(U,\vu) \subset mQ\cap\Z^n\subseteq \gapP(U,c\,\vu).
  \end{equation*}
The polytopes $\gapP_\R(U,\lfloor \vu\rfloor )$ and $mQ$ are integral  and  so we
get  
\begin{equation}
\begin{split} 
   \gapP_\R(U,\lfloor \vu\rfloor ) & = \conv( \gapP_\R(U,\lfloor
   \vu\rfloor )\cap \Z^n)=\conv \gapP(U,\lfloor
   \vu\rfloor )\\ &  \subseteq \conv \gapP(U, \vu) \subseteq
   \conv(mQ\cap\Z^n)=mQ\subseteq mK.
\end{split} 
\label{eq:first}
\end{equation}
Since $mQ$ integral we have $mQ\subseteq \gapP_\R(U,c\,\vu)$ and 
due to Lemma \ref{lem:1} we know for the entries of $\vu$ that
$u_i\geq 1+c/\epsilon$, $1\leq i\leq n$, which implies that
\begin{equation*}
   \frac{u_i}{\lfloor u_i\rfloor}\leq \frac{u_i}{ u_i-1}\leq
   1+\frac{\epsilon}{c} 
\end{equation*}
and thus 
$c\,\vu\leq (c+\epsilon) \lfloor \vu\rfloor$. Hence, 
\begin{equation*}
\begin{split} 
  mQ & =\conv(mQ\cap\Z^n) \subseteq \conv \gapP(U, c\,\vu) \\ &\subseteq 
  \gapP_\R(U, c\,\vu) \subseteq 
  \gapP_\R(U, (c+\epsilon)\,\lfloor\vu\rfloor),
\end{split} 
\end{equation*}
 and with \eqref{eq:first}
\begin{equation} 
      \gapP_\R(U,m^{-1}\lfloor \vu\rfloor) \subseteq   K\subseteq \gapP_\R(U,
      (1+\epsilon)\,(c+\epsilon)\,m^{-1}\lfloor\vu\rfloor).
\end{equation} 
Observe, that actually $U=U_\epsilon,\vu=\vu_\epsilon$ as well as 
$m=m_\epsilon$ depend on the chosen $\epsilon$. 
Now, since $K$ is bounded and  all entries of $U$ are integral, the first inclusion shows that the sequence
$m_\epsilon^{-1}\lfloor \vu_\epsilon\rfloor$, $\epsilon>0$, has to be
bounded.  Therefore, we may assume that 
it converges to $\ov\vu$ as $\epsilon$ approaches $0$. Next, let us
assume that  a
sequence of a (fixed) column vector of the unimodular matrices
$U_\epsilon$ is  unbounded. Since
$\vol(\gapP_\R(U_\epsilon,\vone))=2^n$ and since
$m_\epsilon^{-1}\lfloor \vu_\epsilon\rfloor$ is bounded this shows
that the inradius of $\gapP_\R(U_\epsilon,
      (1+\epsilon)\,(c+\epsilon)\,m_\epsilon^{-1}\lfloor\vu_\epsilon\rfloor)$ converges
      to $0$ as $\epsilon$ tends to $0$. Hence, also $U_\epsilon$
      converges to an unimodular matrix $\ov U$ and so we have shown 
\begin{equation*}
  \gapP_\R(\ov U,\ov\vu) \subseteq   K\subseteq \gapP_\R(\ov U,c\,\ov\vu). 
\end{equation*} 
For the reverse implication we assume that there exists an unimodular
{\gap} $\gapP(U,\vu)$ fullfiling ii). Then 
\begin{equation*}
   \gapP_\R(U,\vu)\cap\Z^n \subseteq   K\cap\Z^n \subseteq
   \gapP_\R(U,c\,\vu)\cap\Z^n, 
\end{equation*}
and by the unimodularity of $U$ we have $\gapP_\R(U,\vu)\cap\Z^n =
\gapP(U,\vu)$ as well as $\gapP_\R(U,c\,\vu)\cap\Z^n = \gapP(U,c\,\vu)$.
\end{proof}

We close this section with  lower bounds on the factors in
\eqref{eq:inclusion} and \eqref{eq:cardinality} of Theorem \ref{thm:discretejohnimproved}.

\newpage
\begin{proposition}\hfill
\label{prop:lowerbounds}  
\begin{enumerate} 
\item 
Let $\tau=\tau(n)\in\R_{>0}$ be a constant depending on
  $n$ such that for every $K\in\Ksn$ there exists a {\gap} $\gapP(A,\vb)\subset
      K$ such that $K\cap\Z^n\subseteq \gapP(A,\tau\,\vb)$. Then $\tau\geq
      n!^{1/n}>\frac{1}{\mathrm e} n$. 
\item  Let $\nu=\nu(n)\in\R_{>0}$ be a constant depending on
  $n$ such that for every $K\in\Ksn$ there exists a {\gap}
  $\gapP(A,\vb)\subset 
      K$ such that $|K\cap\Z^n|\leq \nu\, |\gapP(A,\vb)|$. Then $\nu\geq
      (2^n+1)/3$. 
\end{enumerate}
\end{proposition} 
\begin{proof} For i) we consider for an integer $m\in\N$ the
  cross-polytope $m\dual{C_n}=\{\vx\in\R^n : |x_1|+\cdots +|x_n|\leq
  m\}$ and let $\gapP(U,\vu)$ be a {\gap} such that
  \begin{equation}
    \label{eq:2}
    \gapP(U,\vu)\subseteq m\dual{C_n}\cap\Z^n\subseteq  \gapP(U,\tau\,\vu).
  \end{equation}
In view of Proposition \ref{prop:1}, or since $m\dual{C_n}$ contains the
unit vectors $\ve_1,\dots,\ve_n$ we have $U\in\GL(n,\Z)$. Moreover,
since $m\ve_i\in m\dual{C_n}$, $1\leq i\leq n$, we get from the second
inclusion in \eqref{eq:2} and Lemma \ref{lem:1} that $m\leq \tau\, u_i$, $1\leq
i\leq n$,  and so
\begin{equation}
  \vol(m\dual{C_n})=m^n\frac{2^n}{n!} \leq
  \tau^n\frac{2^n}{n!}\prod_{i=1}^n u_i.
  \label{eq:3} 
\end{equation}
On the other hand, the first inclusion in  \eqref{eq:2} implies
\begin{equation*}
  \gapP_\R(U,\lfloor \vu\rfloor)=\conv\gapP(U,\vu)\subseteq m\dual{C_n},
\end{equation*}
and so
\begin{equation*}
  2^n\prod_{i=1}^n\lfloor u_i\rfloor =\vol \gapP_\R(U,\lfloor
  \vu\rfloor) \leq  \vol(m\dual{C_n}).
\end{equation*}
Combined with \eqref{eq:3} we obtain
\begin{equation*}
  \tau\geq n!^{1/n}\left( \prod_{i=1}^n \frac{\lfloor
      u_i\rfloor}{u_i}\right)^{1/n}. 
\end{equation*}
This is true for any $m\in\N$, and since $u_i\to\infty$ for
$m\to\infty$, we find $\tau \geq  n!^{1/n} > \frac{n}{\mathrm{e}}$.
In order to prove ii), let $Q$ be the $o$-symmetric lattice polytope
given by $Q=\conv( \pm ([0,1]^{n-1}\times\{1\}))$. Then it is easy to
see that $Q\cap
\Z^n=\pm (\{0,1\}^{n-1}\times\{1\})\cup\{\vnull\}$ and hence,  $Q$
does not contain $\vx,\vy\in\Z^n\setminus\{\vnull\}$, $\vx\ne\vy$, and
$\vx+\vy\in Q$. Thus for any {\gap} $\gapP(A,\vb)\subset Q$ we have
$|\gapP(A,\vb)|\leq 3$ and so
\begin{equation*}
  2^{n}+1 =|Q\cap\Z^n| \leq \nu\,|\gapP(A,\vb)|\leq 3\,\nu 
\end{equation*}
yielding the desired lower bound.
\end{proof} 
\section{Proofs of the theorems}
For the proof of the inclusion bound \eqref{eq:inclusion} of Theorem
\ref{thm:discretejohnimproved} we follow essentially the proof of
\cite{Tao:2008ul}, but we will apply a different lattice reduction
taking into account also the polar lattice. More precisely, for a
lattice $\Lambda\in\Lat^n$ with basis $B=(\vb_1,\dots,\vb_n)$, i.e., $\Lambda=B\Z^n$, we
denote by
\begin{equation*}
  \dual\Lambda=\{\vy\in\R^n : \ip{\vx}{\vy}\in\Z\text{ for all }\vx\in\Lambda\}=B^{-\trans}\Z^n
\end{equation*}
its polar lattice. In particular, if  $B^{-T}=(\dual{\vb_1},\dots,\dual{\vb_n})$,
then
\begin{equation}
  \label{eq:4_1}
   \ip{\dual{\vb_i}}{\vb_j}=\delta_{i,j},
\end{equation}
where $\delta_{i,j}$ denotes the Kronecker-symbol. Now a basis $B$ of a
lattice $\Lambda$ is called Seysen reduced if
\begin{equation*}
  \label{eq:4_2}
      S(B)=\sum_{i=1}^n \Vert \vb_i\Vert^2\Vert\dual{\vb_i}\Vert^2
\end{equation*}
is minimal among all bases of $\Lambda$ (cf.~\cite{Seysen1993}). Here,
$\Vert\cdot\Vert$ denotes the Euclidean norm. Seysen proved
\begin{theorem}[\protect{\cite[Theorem 7]{Seysen1993}}] Let $\Lambda\in\Lat^n$. There exists a basis
  $B=(\vb_1,\dots,\vb_n)$ of $\Lambda$ such that $S(B)\leq
  n^{O(\ln n)}$. In particular, for $1\leq i\leq n$ 
  \begin{equation}
    \Vert \vb_i\Vert\,\Vert\dual{\vb_i}\Vert\leq  n^{O(\ln n)}.
    \label{eq:seysen}
\end{equation} 
 \label{thm:seysen}
\end{theorem} 
For an explicit bound we refer to \cite{Maze:2010dm} and for more
information on lattice reduction and Geometry of Numbers  we refer to
\cite{Gruber:1987vp}, \cite{Cassels:1971p7266}. For sake of 
comprehensibility we split the proof of Theorem \ref{thm:discretejohnimproved}
into two parts, one covering the inclusion bound and one the
cardinality bound.

\begin{proof}[Proof of i) of Theorem \ref{thm:discretejohnimproved}]
 In view of John's theorem \eqref{eq:inclusion_john} we may apply a linear
 transformation $T$ to $K$ such that with $\tilde K= T\,K$  
 \begin{equation}
   \label{eq:4_6}
             \ban\subseteq \tilde K\subseteq\sqrt{n}\ban. 
  \end{equation}
  With $\Lambda=T\Z^n$  the
  problem is now to find a {\gap} $\gapP(A,\vb)$ in $\Lambda$ such that
  $\gapP(A,\vb)\subset \tilde K$ and 
  \begin{equation*}
    \tilde K\cap\Lambda\subset\gapP(A,n^{O(\ln n)}\vb).
  \end{equation*}
  Let $B=(\vb_1,\dots,\vb_n)$ be a Seysen reduced basis of $\Lambda$ 
  with associated basis
  $B^{-\trans}=(\dual{\vb_1},\dots,\dual{\vb_n})$ of the polar lattice and let $\vu\in\R^n$ be given by
  $u_i=(1/n)\Vert \vb_i\Vert^{-1}$, $1\leq i\leq n$.

  First, for $\vx\in \gapP_\R(B,\vu)$ we have
  $\vx=\sum_{i=1}^n\lambda_i\vb_i$ with $|\lambda_i|\leq u_i$ and by the
  triangle inequality we conclude $\Vert\vx\Vert\leq 1$. Hence, with
  \eqref{eq:4_6} we certainly have  $\gapP(B,\vu)\subset \tilde K$.
  On the other hand, given $\vx=\sum_{i=1}^n \beta_i\vb_i\in \tilde K$ we
  get by Cramer's rule and \eqref{eq:4_6}
  \begin{equation*}
    |\beta_i|=\frac{|\det(\vx,\vb_1,\dots,\vb_{i-1},\vb_{i+1},\vb_n)|}{|\det
      B|} \leq \sqrt{n}\frac{\vol_{n-1}(\vb_1,\dots,\vb_{i-1},\vb_{i+1},\vb_n)}{\vol(\vb_1,\dots,\vb_n)},
  \end{equation*}
  where $\vol_k(\vc_1,\dots,\vc_k)$ denotes the $k$-dimensional volume
  of the parallelepiped $\{\sum_{i=1}^k \mu_i\vc_i: 0\leq \mu_i\leq
  1\}$. By \eqref{eq:4_1} we find that
  \begin{equation*}
    \begin{split} 
      \vol(\vb_1,\dots,\vb_n)&=\vol_{n-1}(\vb_1,\dots,\vb_{i-1},\vb_{i+1},\vb_n)
      \frac{ \ip{\dual{\vb_i}}{\vb_i}}{\Vert \dual{\vb_i}\Vert}\\
      &= \vol_{n-1}(\vb_1,\dots,\vb_{i-1},\vb_{i+1},\vb_n)
      \frac{1}{\Vert \dual{\vb_i}\Vert},
    \end{split} 
  \end{equation*}
  and thus for $1\leq i\leq n$
  \begin{equation}
    |\beta_i|\leq \sqrt{n}\Vert \dual{\vb_i}\Vert.
\label{eq:bound1}    
         \end{equation}
 Together with  the definition of $u_i$ and Seysen's bound
 \eqref{eq:seysen} we conclude $|\beta_i|\leq n^{3/2}n^{O(\ln n)}u_i$,
 $1\leq i\leq n$. Hence,
 \begin{equation*}
           \tilde K\cap\Lambda\subseteq \gapP_\R(B, n^{O(\ln
             n)}\vu)\cap\Lambda = \gapP(B, n^{O(\ln
             n)}\vu),
 \end{equation*}
 since $B$ is a basis of $\Lambda$.
\end{proof} 

\begin{remark} The optimal upper bound in  Theorem
 \ref{thm:seysen} for a Seysen reduced basis is not known, but any
 improvement on this bound would immediately yield an improvement of 
\eqref{eq:inclusion}.
\end{remark}

For the cardinality bound \eqref{eq:cardinality} of Theorem
\ref{thm:discretejohnimproved} we need another tool from Geometry
  of Numbers, namely Minkowksi's successive minima $\lambda_i(K,\Lambda)$,
    which for $K\in\Ksn$, $\Lambda\in\Lat^n$ and  $1\leq i\leq n$ are defined by
  \begin{equation*}
    \lambda_i(K,\Lambda)=\min\{\lambda>0 : \dim(\lambda K\cap\Lambda)\geq i\}. 
  \end{equation*}
In words, $\lambda_i(K,\Lambda)$ is the smallest dilation factor $\Lambda$
such that $\lambda\, K$ contains $i$ linearly independent lattice
points of $\Lambda$. Minkowski's fundamental second theorem on successive minima
states that \cite[\S 9, Theorem 1]{Gruber:1987vp}) 
\begin{equation}
    \vol(K)\leq \det\Lambda \prod_{i=1}^n \frac{2}{\lambda_i(K,\Lambda)},
  \label{eq:min_volume}
\end{equation}
and here we 
need a discrete version of it. In \cite{Henk:2002vta} it was shown that
\begin{equation}
  \label{eq:discreteminkowski}
  |K\cap\Lambda|\leq 2^{n-1}\prod_{i=1}^n \left\lfloor\frac{2}{\lambda_i(K,\Lambda)}+1\right\rfloor,
\end{equation}
and for an improvement on the constant $2^{n-1}$ and
related results we refer to \cite{Malikiosis:2010tb,
  Malikiosis:2010vi}. It is conjectured in \cite{Betke:1993jn} that
\eqref{eq:discreteminkowski} holds without any additional factor in front of the product which
would, in particular, imply Minkowski's volume bound.

\begin{proof}[Proof of ii) of Theorem \ref{thm:discretejohnimproved}]
  Let $\va_i\in\Z^n$, $1\leq i\leq n$, be  linearly independent
  lattice vectors corresponding to the successive minima
  $\lambda_i=\lambda_i(K,\Z^n)$, i.e., $\va_i\in\lambda_i\,K$, $1\leq i\leq
  n$. Since $\lambda_i^{-1}\va_i\in K$ it follows
  \begin{equation*}
    \left\{\sum_{i=1}^n\mu_i\frac{1}{n\lambda_i}\va_i :
      -1\leq\mu_i\leq 1\right\}\subset \conv\{\pm\lambda_i^{-1}\va_i
    :1\leq i\leq n\}\subseteq K.
  \end{equation*}
  Thus, denoting by $A$ the matrix with columns $\va_i$ and letting
  $\vb$ be the vector with entries $b_i=(n\lambda_i)^{-1}$ we have
  $\gapP(A,\vb)\subset K$ and
  \begin{equation*}
     |P(A,\vb)|=\prod_{i=1}^n \left(2\left\lfloor\frac{1}{n\lambda_i}\right\rfloor +1\right).
   \end{equation*}
   Now it is not hard to see that
   \begin{equation}
     2\left\lfloor\frac{1}{n\lambda_i}\right\rfloor +1 \geq
     \frac{1}{3}\frac{1}{n}\left\lfloor\frac{2}{\lambda_i}+1\right\rfloor
\label{eq:second} 
   \end{equation}
   and with \eqref{eq:discreteminkowski} we get
   \begin{equation*}
     |P(A,\vb)| \geq
     \left(\frac{1}{3n}\right)^n\left(\frac{1}{2}\right)^{n-1}2^{n-1}\prod_{i=1}^n
     \left\lfloor\frac{2}{\lambda_i}+1\right\rfloor> (6n)^{-n}|K\cap\Z^n|.
   \end{equation*}
   This shows \eqref{eq:cardinality}.
\end{proof}

\begin{remark} 
  We want to point out that the columns of the matrix $A$ of the {\gap} in the above proof of
    the cardinality bound  of Theorem \ref{thm:discretejohnimproved}
    do not build a basis of $\Z^n$ (in general) 
    and hence, this {\gap} cannot be used in order to obtain an inclusion bound. 
\end{remark}

Now the proof of Theorem \ref{thm:onegap} is a kind of combination of
the two proofs leading to \eqref{eq:inclusion} and
\eqref{eq:cardinality}. Instead of a Seysen reduced basis we exploit
properties of a so called Hermite-Korkine-Zolotarev (HKZ) reduced basis
$\vb_1,\dots,\vb_n$  of the lattice $\Lambda$. For such a basis it was shown by Mahler (see, e.g., \cite[Theorem 2.1]{Lagarias:1990p675})
that for $1\leq i\leq n$ 
\begin{equation}  
 \Vert\vb_i\Vert\leq \frac{\sqrt{{i+3}}}{2}\lambda_i(B_n,\Lambda). 
\label{eq:hkzbasis} 
\end{equation} 
Moreover, 
H{\r a}stad\&Lagarias \cite{Hastad:1990tk} pointed out that for such a
HKZ-basis one has  
\begin{equation} 
  \Vert b_i\Vert\cdot\Vert\dual{b}_i\Vert\leq
  \left(\frac{3}{2}\right)^n<n^{\frac{1}{2}n/\ln n}.
\label{eq:hkzpolar} 
\end{equation}   
This bound is  worse than the one given  in \eqref{eq:seysen}, but
the advantage of a HKZ reduced basis is its close relation to the 
successive minima \eqref{eq:hkzbasis}.

\begin{proof}[Proof of Theorem \ref{thm:onegap}] First we may assume
  that $\lambda_n(K,\Z^n)\leq 1$, i.e., that $K$ contains $n$ linearly
  independent lattice points. Otherwise, all lattice points of $K$
  lying in a hyperplane $H$ and it would be sufficient to prove the
  theorem with respect to the $n-1$-dimensional convex body $K\cap H$
  and lattice $H\cap\Z^n$.  

  Now we proceed completely
  analogously to the proof
  of i) in Theorem \ref{thm:discretejohnimproved}; we just replace the Seysen reduced basis by
  a HKZ-reduced basis $B=(\vb_1,\dots,\vb_n)$,
  and the {\gap} is given by $\gapP(B,\vu)$ with
  $u_i=(1/n)\Vert\vb_i\Vert^{-1}$, $1\leq  i\leq  n$.  Replacing
  \eqref{eq:seysen} by \eqref{eq:hkzpolar} in \eqref{eq:bound1} leads
  then to
  \begin{equation*}
  \gapP(B,\vu)  \subseteq \tilde K\cap\Lambda\subseteq  \gapP(B, n^{O(n/\ln
             n)}\vu),
  \end{equation*} 
  where $\tilde K$ was a linear image of $K$ such that
\begin{equation}
   \label{eq:4_6_2}
             \ban\subseteq \tilde K\subseteq\sqrt{n}\ban. 
\end{equation}
It remains to prove the cardinality bound for the {\gap} $\gapP(B,\vu)$
and $\tilde K$. Regarding the size of $\gapP(B,\vu)$ we have
\begin{equation}
  |\gapP(B,\vu)| =\prod_{i=1}^n \left(2\left\lfloor
      \frac{1}{n\Vert\vb_i\Vert}\right\rfloor+1\right)\geq
  n^{-n}\prod_{i=1}^n \frac{1}{\Vert\vb_i\Vert}.
  \label{eq:bound2} 
\end{equation} 
On the other hand, for an upper bound on $\tilde K\cap\Lambda$ we use
\eqref{eq:discreteminkowski} and since $\lambda_n(K,\Lambda)\leq 1$ we
get
\begin{equation*}
  |K\cap\Lambda|\leq 2^{n-1} \prod_{i=1}^n
  \left(\frac{2}{\lambda_i(K,\Lambda)}+1\right) \leq 6^n
  \prod_{i=1}^n\frac{1}{\lambda_i(K,\Lambda)}.
\end{equation*}
In view of \eqref{eq:4_6_2} and \eqref{eq:hkzbasis} we obtain
\begin{equation*}
  |K\cap\Lambda|\leq 6^n
  \prod_{i=1}^n\frac{1}{\lambda_i(\sqrt{n}\ban,\Lambda)} = (6\sqrt{n})^n
  \prod_{i=1}^n\frac{1}{\lambda_i(\ban,\Lambda)}\leq (6n)^n \prod_{i=1}^n\frac{1}{\Vert\vb_i\Vert}.
\end{equation*}   
Combined with \eqref{eq:bound2} we get  $|K\cap\Lambda|\leq O(n)^{2n}  |\gapP(B,\vu)|$.

\end{proof}

Finally, we consider unconditional bodies $K\in\Ksn$, i.e., bodies which
are symmetric to all coordinate hyperplanes. As stated  in Proposition
\ref{prop:unconditional}, in this special case the inclusion bound can
be made linear in the dimension. In view of Proposition \ref{prop:lowerbounds} this is
also the optimal order within this class of bodies as the given
example used for the lower bound in Proposition \ref{prop:lowerbounds} is
unconditional.

\begin{proof}[Proof of Proposition \ref{prop:unconditional}] For $i=1,\dots,n$ let
  $u_i$ be the maximal entry of the $i$th coordinate of a point of
  $K$. Then $u_i>0$ and
  \begin{equation}
    K\cap\Z^n\subseteq \gapP(I_n,\vu)
 \label{eq:bound3}   
  \end{equation}
with $\vu=(u_1,\dots,u_n)^\trans$ and $I_n$ the $n\times n$-identity matrix. By the unconditionality of $K$ we
have $\pm u_i\,\ve_i\in K$, $1\leq i\leq n$, and thus
\begin{equation*}
  \gapP_\R(I_n,n^{-1}\vu)\subset\conv\{\pm u_i\ve_i: 1\leq i\leq
  n\}\subseteq K.
\end{equation*}
Hence, $\gapP(I_n,n^{-1}\vu)\subset K$. For the remaining cardinality
bound we observe that $(2\,u_i+1)< (2\lfloor u_i/n\rfloor)+1)\,3n$ and
so \eqref{eq:bound3} implies
\begin{equation*}
  \begin{split}
  |K\cap\Z^n|& \leq \prod_{i=1}^n (2\,\lfloor u_i\rfloor +1)< (3n)^n
  \prod_{i=1}^n (2\lfloor u_i/n\rfloor)+1)\\  & =(3n)^n
  |\gapP(I_n,n^{-1}\vu)|.
\end{split}  
\end{equation*}

\end{proof}


\begin{thebibliography}{AAGM15}

\bibitem[AAGM15]{ArtsteinAvidan:2015hq}
Sh.~Artstein-Avidan, A.~Giannopoulos, and V.~Milman.
\newblock {\em {Asymptotic Geometric Analysis, Part I}}, volume 202.
\newblock AMS, 2015.


\bibitem[BHW93]{Betke:1993jn}
U.~Betke, M.~Henk, and J.M.~Wills.
\newblock {Successive-minima-type inequalities}.
\newblock {\em Discrete {\&} Computational Geometry}, 9(2):165--175, 1993.

\bibitem[BV92]{Barany:1992p12355}
I.~B\'ar\'any and A.M.~Vershik.
\newblock {On the number of convex lattice polytopes}.
\newblock {\em Geometric Functional Analysis}, 2(4):381--393, 1992.

\bibitem[Cas71]{Cassels:1971p7266}
J.W.S.~Cassels.
\newblock {\em {An introduction to the geometry of numbers}}.
\newblock Springer, 1971.

\bibitem[GL87]{Gruber:1987vp}
P.M.~Gruber and C.G.~Lekkerkerker.
\newblock {\em {Geometry of numbers}}, 
\newblock North-Holland, second edition, 1987.

\bibitem[Hen02]{Henk:2002vta}
M.~Henk.
\newblock {Successive minima and lattice points}.
\newblock {\em Rend. Circ. Mat. Palermo (2) Suppl.}, (70, part I):377--384,
  2002.

\bibitem[HL90]{Hastad:1990tk}
J.~H{\aa}stad and J.C.~ Lagarias.
\newblock {Simultaneously good bases of a lattice and its reciprocal lattice}.
\newblock {\em Mathematische Annalen}, 287:163--174, 1990.

\bibitem[LLS90]{Lagarias:1990p675}
J.C.~Lagarias, H.W.~Lenstra Jr., and C.P.~Schnorr.
\newblock {Korkin-Zolotarev bases and successive minima of a lattice and its
  reciprocal lattice}.
\newblock {\em Combinatorica}, 10(4):333--348, 1990.

\bibitem[Mal10]{Malikiosis:2010tb}
R.~D.~Malikiosis.
\newblock {An Optimization Problem related to Minkowski's successive minima}.
\newblock {\em Discrete {\&} Computational Geometry}, 43(4):784--797, 2010.

\bibitem[Mal12]{Malikiosis:2010vi}
R.~D.~Malikiosis.
\newblock {A discrete analogue for Minkowski's second theorem on successive
  minima}.
\newblock {\em Advances in Geometry}, 12(2):365--380, 2012.

\bibitem[Maz10]{Maze:2010dm}
G.~Maze.
\newblock {Some inequlaities related to the Seysen measure of a lattice}.
\newblock {\em Linear Algebra and its Applications},
433(8-10):1659--1665, 2010.

\bibitem[Sch14]{Schneider:2014td}
R.~Schneider.
\newblock {\em {Convex Bodies: The Brunn-Minkowski Theory}}.
\newblock Cambridge University Press, second expanded edition edition, 2014.

\bibitem[Sey93]{Seysen1993}
M~Seysen.
\newblock {Simultaneous reduction of a lattice basis and its reciprocal basis}.
\newblock {\em Combinatorica}, 13(3):363--376, 1993.

\bibitem[TV06]{Tao:2006vs}
T.~Tao and V.H.~ Vu.
\newblock {\em {Additive combinatorics }}.
\newblock Cambridge University Press, 2006.

\bibitem[TV08]{Tao:2008ul}
T.~Tao and V.H.~Vu.
\newblock {John-type theorems for generalized arithmetic progressions and
  iterated sumsets}.
\newblock {\em Advances in Mathematics}, 219(2):428--449, 2008.

\end{thebibliography}

\end{document}